\documentclass[letterpaper,10pt,conference]{IEEEtran}

\IEEEoverridecommandlockouts
\usepackage{cite}
\usepackage{amsmath,amsthm,amssymb,array,cases,mathrsfs,subcaption,here,enumitem,ascmac,url,cite,float,comment,color}
\usepackage{algorithmic}
\usepackage{graphicx}
\usepackage{textcomp}
\usepackage{xcolor}

\theoremstyle{plain}
\newtheorem{theorem}{Theorem}
\newtheorem{lemma}{Lemma}
\newtheorem{proposition}{Proposition}
\newtheorem{corollary}{Corollary}
\newtheorem{remark}{Remark}
\newtheorem{assumption}{Assumption}

\theoremstyle{definition}
\newtheorem{definition}{Definition}

\usepackage{hyperref}

\begin{document}

\title{\LARGE
Global Linearization of Parameterized Nonlinear Systems \\ with Stable Equilibrium Point Using the Koopman Operator \\
\thanks{The work was partially supported by JSPS KAKENHI (Grant No. 23K03914), JSPS Bilateral Collaborations (Grant No. JPJSBP120242202), and JST BOOST (Grant No. JPMJBS2407). 
}
}

\author{Natsuki~Katayama$^1$\thanks{$^1$Natsuki Katayama and Yoshihiko Susuki are with the Department of Electrical Engineering, Kyoto University, Katsura, Nishikyo-ku, Kyoto 615-8510, Japan. E-mails: \texttt{n-katayama@dove.kuee.kyoto-u.ac.jp}, \texttt{susuki.yoshihiko.5c@kyoto-u.ac.jp}}, Alexandre Mauroy$^2$\thanks{$^2$Alexandre Mauroy is with the Department of Mathematics, University of Namur, 5000, Belgium. E-mails: \texttt{alexandre.mauroy@unamur.be}}, and Yoshihiko~Susuki$^1$}

\maketitle

\begin{abstract}
The Koopman operator framework enables global analysis of nonlinear systems through its inherent linearity.
This study aims to clarify spectral properties of the Koopman operators for nonlinear systems with control inputs. 
To this end, we treat the inputs as parameters throughout this paper. 
We then introduce the Koopman operator for a parameterized dynamical system with a globally exponentially stable equilibrium point and analyze how eigenfunctions of the operator depend on the parameter. 
As a main result, we obtain a global linearization, which enables one to transform the nonlinear system into a finite-dimensional linear system, and we show that it depends continuously on the parameter. 
Subsequently, for a control-affine system, we investigate a condition under which the transformation providing a global bilinearization does not depend on the parameter. 
This provides the condition under which the global bilinearization for the control-affine system is independent of the parameter.  
\end{abstract}

\IEEEpeerreviewmaketitle

%%%%%
\section{Introduction}
\label{sec:1}
The Koopman operator has gained significant attention in recent years for the analysis and control of nonlinear systems \cite{mauroy2020koopman}. 
Originally introduced for finite-dimensional nonlinear autonomous systems, the Koopman operator is a linear operator acting on a Banach space.
The dynamics on the finite-dimensional state space can thus be equivalently described by the action of the Koopman operator on this (infinite-dimensional) Banach space, allowing the nonlinear state dynamics %global analysis 
using the linear %system 
theory for identification and prediction.
In particular, recent intensive studies (e.g., \cite{brunton2016koopman}) have developed data-driven algorithms to learn the Koopman operator, such as Dynamic Mode Decomposition (DMD) and its variants. 
The so-called Extended DMD (EDMD) \cite{williams2015data} can be applied even to complex, black box models when sufficient data are available. 
Furthermore, for nonlinear control systems, research has been conducted on linearization (both in the state and the input) \cite{proctor2016dynamic,proctor2018generalizing}, bilinearization \cite{goswami2021bilinearization,peitz2020data,bruder2021advantages}, and more general formulations \cite{haseli2023modeling, guo2025learning}, all of which provide finite-dimensional representations (see the recent overview \cite{strasser2025overview}).

It should be noted that these representations fundamentally rely on %the 
spectral properties of the Koopman operator. 
%Assumptions about the 
The existence of finite-dimensional representations depends on the existence of %the 
eigenfunctions of the Koopman operator, called Koopman eigenfunctions \cite{mauroy2020koopman}. 
This implies that, while the Koopman operator itself is defined on an infinite-dimensional Banach space, a finite number of Koopman eigenfunctions can construct an embedding from the state space to a finite-dimensional space, yielding a \emph{global linearization}. 
{\color{black}Thereby, the nonlinear system can be transformed into a finite-dimensional linear system.} 
The spectral properties for autonomous dynamical systems, particularly the existence and uniqueness of Koopman eigenfunctions, have been studied extensively \cite{mezic2005spectral,mezic2013analysis,mezic2020spectrum,kvalheim2021existence}, as well as %on 
the existence of global linearizations constructed from these eigenfunctions \cite{liu2023non,kvalheim2023linearizability,liu2025properties}. 
However, such linearizations for nonlinear systems with inputs have not yet been fully clarified (although Koopman-based feedback linearization inspired by geometric nonlinear control theory has been studied \cite{gadginmath2024data}). 
{\color{black}If a similar theory is established for the nonlinear systems with inputs, then it will provide a theoretical foundation of data-driven controller design of nonlinear systems via EDMD and its variants, which has assumed the existence of finite-dimensional representations.}

This paper aims to analyze spectral properties of the Koopman operator for nonlinear control systems.
Specifically, by treating an input as a (piecewise) constant parameter, we study how Koopman eigenfunctions depend on the parameters, clarifying the existence of global linearization for the control systems. 
%This work provides a theoretical foundation for the validity of EDMD-like algorithms that assume the existence of finite-dimensional representations. 
The contents of this paper are as follows: 
We consider a dynamical system $\dot{x} = F^u(x)$ with an input $u$ regarded as a parameter and assume that the system has a globally exponentially stable equilibrium point (GES EP). 
This enables us to leverage the spectral properties of the Koopman operator established by \cite{kvalheim2021existence}. 
We then analyze the parameter-dependence of the Koopman operator and of the Koopman eigenfunctions using the perturbation theory of linear operators  \cite{kato2013perturbation}. 
This analysis leads to the existence of a global linearization $\psi^u$ for the parameterized system that depends continuously on $u$ (Theorem \ref{thm:linearization}).
We next consider the condition under which $\psi^u$ does not depend on $u$, focusing on a control-affine system. 
This consideration, technically based on Lie algebraic arguments, results in a global bilinearization of the control-affine system that is independent of $u$ (Theorem \ref{thm:bilinearization}). 
%Throughout this paper we focus on the system with GES EP, however, we believe that the linearization theorems can be extended to a globally exponentially stabilizable system by considering an input change $v = k(x,u)$. 

The remainder of this paper is organized as follows.
The rest of Section~\ref{sec:1} provides preliminaries.
Section~\ref{sec:2} analyzes the resolvent of the Koopman operator.
Section~\ref{sec:3} defines the Koopman operators for a parameterized dynamical system and analyzes the parameter-dependence of its eigenfunctions. 
This enables the global linearization that depends continuously on the parameter.
Section~\ref{sec:4} presents the global bilinearization of a control-affine system, and Section~\ref{sec:5} concludes the paper.

\subsection{Preliminary}
\subsubsection*{Notations}
We denote the set of real numbers by $\mathbb{R}$, the set of $n$-dimensional real numbers by $\mathbb{R}^n$, the set of complex numbers by $\mathbb{C}$, the set of $n$-dimensional complex numbers by $\mathbb{C}^n$, the set of non-negative integers by $\mathbb{N}_{\geq 0}$, {\color{black}the set of natural numbers by $\mathbb{N}_{\geq 1}$,} and the set of $n$-tuples of non-negative integers by $\mathbb{N}_{\geq 0}^n$. 
The dual space of $\mathbb{R}^n$ is denoted by $(\mathbb{R}^n)^*$ and that of $\mathbb{C}^n$ by $(\mathbb{C}^n)^*$. 
For the $d$-dimensional Euclidean space, a closed ball with a center $u_0$ and a radius $\delta$ is denoted by $\bar{B}^{\mathbb{R}^d} (u_0,\delta)$.
For Banach spaces $\mathcal{F}$ and $\mathcal{G}$, we denote the set of bounded linear operators from $\mathcal{F}$ to $\mathcal{G}$ by $\mathcal{L} (\mathcal{F}, \mathcal{G})$, and the set of bounded, symmetric multilinear operators from $\mathcal{F}^{\otimes i}$ to $\mathcal{G}$ by $\mathcal{L}^i (\mathcal{F}, \mathcal{G})$. 
For a differentiable map $F:\mathbb{K}^n \to \mathbb{K}^m$ with $\mathbb{K} = \mathbb{R}$ or $\mathbb{C}$, ${\rm D}_x^i F (x_0)$ with $x_0 \in \mathbb{K}^n$ represents an $i$-th order derivative, which is an element of $\mathcal{L}^i (\mathbb{K}^n,\mathbb{K}^m)$.

\subsubsection*{Linear Operators}
Let $\mathcal{F}$ be a Banach space and $L:\mathcal{D}(L) \to \mathcal{F}$ be a closed linear operator with a domain $\mathcal{D} (L)$. 
The \emph{resolvent set} of $L$, denoted by $\rho(L) \subset \mathbb{C}$, is the set of all complex numbers $\zeta$ for which $L -\zeta I$ is invertible and $(L - \zeta I)^{-1}$ is bounded, where $I$ is the identity operator. 
The \emph{resolvent} of $L$, denoted by $R(\zeta,L):\mathcal{F} \to \mathcal{F}$, is a bounded operator defined by
$$ R(\zeta,L) = (L -\zeta I)^{-1}, \quad \forall \zeta \in \rho (L). $$
The complement set of $\rho(L)$, denoted by $\sigma(L)$, is called spectrum, containing the eigenvalues of $L$.

\begin{definition}
An eigenvalue $\mu \in \sigma(L)$ is said to be \emph{simple} if the generalized eigenspace of $L$ associated with $\mu$ is one-dimensional. 
\end{definition}
\begin{definition}
An eigenvalue $\mu \in \sigma(L)$ is said to be \emph{isolated} if there exists $\delta > 0$ such that any $\zeta \in \mathbb{C}$ with $0<|\zeta-\mu| < \delta$ belongs to $\rho (L)$. 
\end{definition}
Suppose that $\mu \in \sigma(L)$ is a simple, isolated eigenvalue.
According to \cite{kato2013perturbation} (see also \cite{susuki2021koopman_resolvent, susuki2024koopman}), the projection operator onto the eigenspace of $\mu$, called the eigenprojection, and denoted by $P_\mu$, corresponds to the complex integral of $-R(\zeta,L)/2\pi {\rm i}$:
\begin{equation}
    \label{eq:integral_resolvent}
    P_{\mu}  = -\frac{1}{2\pi {\rm i}} \int_{\gamma} R(\zeta,L){\rm d}\zeta,
\end{equation}
where $\gamma$ is a circle that encloses only $\mu$ and excludes the rest of the spectrum, and on which $R(\zeta,L)$ is holomorphic in $\zeta$. 

Next, the concept of strong convergence of a parameterized linear 
\emph{bounded} operator on a Banach space is introduced from %based on 
\cite{kato2013perturbation}. 
Let $L^u$ be a linear bounded operator with a parameter $u \in \mathbb{U} \subset \mathbb{R}^d$. 
\begin{definition}
For $u _0 \in \mathbb{U}$ with an open set $\mathbb{U} \subset \mathbb{R}^d$, $L^u$ \emph{strongly converges} to $L^{u _0}$ if for any $f\in \mathcal{F}$ and any $\varepsilon$, there exists $\delta>0$ such that 
$$
u  \in \bar{B}^{\mathbb{R}^d}(u _0, \delta) \cap \mathbb{U} \implies \| L^u  f - L^{u _0} f \|_{\mathcal{F}} < \varepsilon. 
$$
\end{definition}

If the resolvent operator $R(\zeta,L^u)$ strongly converges as $u \to u_0$ for all $\zeta \in \gamma$ where $L^u$ is a closed operator, we obtain the strong convergence of the eigenprojection according to the integral \eqref{eq:integral_resolvent}; 
\begin{proposition}[See Chap 8.1.4 of \cite{kato2013perturbation}]
\label{cor:strong_convergence_P}
Assume that, for all $u\in \bar{B}^{\mathbb{R}^d} (u_0, \delta)$ with sufficiently small $\delta > 0$, $\mu^u $ is a simple and isolated eigenvalue of $L^u$, and is continuous as a function of $u $. 
If there exists $\varepsilon >0$ such that $R(\zeta,L^u)$ strongly converges to $R(\zeta,L^{u_0})$ for all $\zeta$ with $0<|\zeta - \mu^{u_0}| < \varepsilon$,  then the eigenprojection of $L^u$ associated with $\mu^u$, denoted by $P_{\mu^u}^u$, strongly converges to $P_{\mu^{u_0}}^{u_0}$. 
\end{proposition}

%%%%%
\section{Resolvent of the Koopman Operator}
\label{sec:2}
In this section, we introduce the Koopman operator for a continuous-time dynamical system and analyze its resolvent. 
Consider a dynamical system
\begin{equation}
    \label{eq:DS_nopara}
    \dot{x} := \frac{d x}{d t} = F(x), \quad x \in \mathbb{X},
\end{equation}
where $\mathbb{X} \subset \mathbb{R}^n$, representing the state space, is the closure of a precompact open set and positively invariant, and $F:\mathbb{X}\to {\rm T}\mathbb{X}$ is a $C^r$ vector field with $r \in \mathbb{N}_{\geq 1}\cup \{\infty \}$. 
Thanks to the positive invariance, the one-parameter semi-group of nonlinear maps, denoted by $\{S_t : \mathbb{X} \to \mathbb{X}\}_{t \geq 0}$, can be defined, which is called a flow. 
We define the family of the Koopman operators, denoted by $\{ U_t :\mathcal{F} \to \mathcal{F} \}_{t \geq 0}$, by a composition operator with the flow as follows:
\begin{equation}
    \label{eq:Koopmandef}
    U_t f = f \circ S_t, \quad f \in \mathcal{F}, 
\end{equation}
where $f:\mathbb{X}\to \mathbb{C}$ is a function called observable and $\mathcal{F}$ is a Banach space. 
It can be shown that $U_t$ is linear, bounded operator for all $t\geq 0$, and is a one-parameter semi-group. 

Here, we consider $\mathcal{F}$ as a set of differentiable functions, denoted by $C^k(\mathbb{X})$ with a positive integer $1\leq k \leq r$, which is a Banach space according to the compactness of $\mathbb{X}$ with the norm
\begin{equation}
    \label{eq:norm}
    \| f \|_{C^k (\mathbb{X})} := \sup_{\substack{0\leq i \leq k \\x\in \mathbb{X}}} \left\| {\rm D}_x^i f (x) \right\|_{\mathcal{L}^i\left( \mathbb{R}^n, \mathbb{C} \right)} < +\infty.
\end{equation}
It then follows that $\{U_t\}_{t\geq 0}$ is a strongly continuous semi-group, implying the existence of the limit
\begin{equation}
    \label{eq:generator}
    \lim_{t \downarrow 0} \frac{U_t f - f}{t} =: L_F f, \quad \forall f \in \mathcal{D} (L_F),
\end{equation}
where the domain $\mathcal{D} (L_F)$ of $L_F$ is a dense set in $\mathcal{F}$. 
We call this infinitesimal generator $L_F$ %as
the Koopman generator associated with the vector field $F$. 
This generator corresponds to a Lie derivative with respect to $F$ and satisfies 
$$
L_F f (x) = {\rm D}_x f (x) \cdot F(x),\quad \forall x \in \mathbb{X}.
$$

Now, we consider the eigenvalue $\mu\in \mathbb{C}$ and the associated eigenfunction $\phi_\mu \in \mathcal{F}\backslash \{0\}$ of $L_F$, called the Koopman eigenvalue and the Koopman eigenfunction, as satisfying 
\begin{equation}
    \label{eq:eigendef}
    L_F \phi_{\mu}= \mu \phi_\mu,\quad \phi_\mu \neq 0,
\end{equation}
or equivalently,
\begin{equation}
    \label{eq:eigendef2}
    U_t \phi_{\mu}= {\rm e}^{\mu t} \phi_\mu,\quad \phi_\mu \neq 0,~\forall t\geq 0.
\end{equation}
Suppose that the system \eqref{eq:DS_nopara} has the globally exponentially stable equilibrium point (GES EP) at the origin $\{ 0 \}$ and that the interior of $\mathbb{X}$ contains the GES EP. 
Denote the Jacobian of $F(x)$ at $x=0$ as ${\rm D}_x F(0)$ and let $\lambda_1,\ldots,\lambda_n \in \mathbb{C}$ 
be the eigenvalues of ${\rm D}_x F(0)$ whose real parts are negative. 
The following proposition clarifies the resolvent set of the Koopman generator. 
\begin{proposition}
\label{prop:resolvent_KO}
Let $k\in \mathbb{N}_{\geq 1}$ be $k \leq r$ and let $\zeta \in \mathbb{C}$ satisfy (i) ${\rm Re}\, \zeta > k\cdot {\rm max}_{1\leq \ell \leq n} {\rm Re} \, \lambda_\ell$ holds; (ii) there is no $(m_1,\ldots,m_n) \in \mathbb{N}_{\geq 0}^n$ such that $\zeta = m_1 \lambda_1 + \cdots + m_n \lambda_n$ and $m_1+ \cdots + m_n \leq k$. 
Then, $\zeta$ is an element of {\color{black}$\rho(L_F)$}. 
\end{proposition}
{\color{black}\begin{proof}
    See Appendix~\ref{app:1}. 
\end{proof}}

According to \cite{kvalheim2021existence}, the eigenvalues of the Jacobian ${\rm D}_x F (0)$ are Koopman eigenvalues whose associated eigenfunctions compose a diffeomorphism representing a conjugate linear system. 
To present this in more detail, we introduce the following two conditions: 
\begin{definition}
    \label{def:k-nonresonant}
For given $k$, $\lambda_i$ with $i\in \{1,\ldots, n\}$ is said to satisfy \emph{$k$-nonresonant condition} if there is no $(m_1,\ldots,m_i,\ldots, m_n) \in \mathbb{N}_{\geq 0}^n \backslash \{ (0,\ldots,1,\ldots,0) \}$ such that $\lambda_i = m_1 \lambda_1 +\cdots + m_n \lambda_n$ and $m_1+\cdots + m_n \leq k$. 
\end{definition}
\begin{definition}
    \label{def:spectral_spread}
For given $k$, $\lambda_i$ with $i\in \{1,\ldots, n\}$ is said to satisfy \emph{$k$-spectral spread condition} if ${\rm Re}\,\lambda_i > k\cdot {\rm max}_{1\leq \ell \leq n} {\rm Re} \, \lambda_\ell$. 
\end{definition}
The authors of \cite{kvalheim2021existence} established the following proposition.
\begin{proposition}[See \cite{kvalheim2021existence}] 
\label{prop:kval}
Let $k\in \mathbb{N}_{\geq 1}$ be $k \leq r$. 
If $\lambda_i$ with $i\in \{1,\ldots, n\}$ satisfy the $k$-nonresonant and $k$-spectral spread conditions, $\lambda_i$ is a simple eigenvalue of $L_F$.
%there uniquely exist $\phi_\mu \in \mathcal{F}\backslash \{0\}$ satisfying $$ U \phi_\mu = \mu \phi_\mu. $$
\end{proposition}
\begin{remark}
\label{rem:isolated_eigenvalue}
Proposition \ref{prop:resolvent_KO} implies that $\lambda_i$ with $i\in \{1,\ldots,n\}$ is an isolated eigenvalue of $L_F$. 
\end{remark}

The Koopman principal eigenfunctions are then defined as $n$-tuple of $\lambda_1, \ldots, \lambda_n$ if these satisfy the $k$-nonresonant and $k$-spectral spread conditions with $k$ appropriately chosen. 
According to \cite{kvalheim2021existence}, the associated eigenfunctions, called Koopman principal eigenfunctions, satisfy ${\rm D}_x \phi_{\lambda_i} (0) \neq 0$ for all $i\in \{1,\ldots,n\}$ and ${\rm span} \{ {\rm D}_x \phi_{\lambda_1} (0),\ldots, {\rm D}_x \phi_{\lambda_n} (0)\} = (\mathbb{R}^n)^*$. 
This fact yields the existence of a $C^k$ diffeomorphsim $\psi : \mathbb{X} \to \mathbb{R}^n$ such that 
$$
 \frac{d}{dt}\psi(x) = {\rm D}_x F(0) \psi (x),\quad \forall x \in \mathbb{X},~\forall t \geq 0
$$
holds (see \cite[Proposition 2]{kvalheim2021existence}). 
{\color{black}
\begin{remark}
\label{rem:real-valued}
The Koopman principal eigenfunctions can be obtained from the eigenprojections of $L_F$, denoted by $P_{\lambda_1}, \ldots, P_{\lambda_n}$, on which our construction is based. 
We also note that if $\lambda_i$ and $\lambda_{i+1}$ are complex conjugate, the associated eigenprojections satisfy 
$$
\text{$f$ is real-valued} \implies \text{$(P_{\lambda_i} + P_{\lambda_{i+1}})f$ is real-valued}.
$$
\end{remark}
}

%%%%%
\section{Koopman Operators for Parameterized Dynamical Systems}
\label{sec:3}

\subsection{Introduction to Parameterized Koopman Operators}
Here, we introduce a dynamical system with a parameter and %the 
associated Koopman operator. 
Consider a continuous-time dynamical system with a parameter $u$, given as %:
\begin{equation}
    \label{eq:paraDS}
    \dot{x} = F^u  (x) ,\quad x\in \mathbb{X},
\end{equation}
{\color{black}where $\mathbb{X} \subset \mathbb{R}^n$ is the state space, $u  = (u _1,\ldots,u _d)\in \mathbb{U}$ is a parameter whose domain $\mathbb{U} \subset \mathbb{R}^d$ is an open set, and $F^u (x)$ is a vector field. }
We make the following two assumptions: %assume the followings:
\begin{assumption}
\label{asm:state_space}
{\color{black}The state space $\mathbb{X}$ is the closure of a precompact open set (and thus compact), and is independent
%\footnote{This is for simplicity. This assumption can be relaxed to the following case: Denote the state space for $F^u $ by $\mathbb{X}^u $. The case assumes that for any $u_1,u_2 \in \mathbb{U}$, $\mathbb{X}^{u_1}$ and $\mathbb{X}^{u_2}$ are $C^r$-diffeomorphic, and that there exists a diffeomorphism $T_{u_1,u_2}:\mathbb{X}^{u _1} \to \mathbb{X}^{u _2}$ such that $T_{u_1,u_2}$ and its (high-order) derivatives up to $r$-th order are continuous in $u_1$ and $u_2$. } 
of $u  \in \mathbb{U}$ and positively invariant for all $u  \in \mathbb{U}$. }
\end{assumption}
\begin{assumption}
\label{asm:continuity_para}
For any fixed $x\in \mathbb{X}$ and any $i \in \{0,\ldots, r\}$, ${\rm D}_x^i F^u  (x)$ are continuous in all $u \in\mathbb{U}$ in the $\mathcal{L}^i(\mathbb{R}^d,\mathbb{R}^n)$ norm. 
\end{assumption}
Assumption \ref{asm:state_space} guarantees that the flow of \eqref{eq:paraDS} parameterized by $u \in \mathbb{U}$ can be defined for all $u\in \mathbb{U}$, denoted by $\{S_t^u:\mathbb{X} \to \mathbb{X}\}_{t\geq 0}$. 
Then, the Koopman operator parameterized by $u \in \mathbb{U}$ is defined by 
\begin{equation}
    \label{eq:paraKoopmandef}
    U_t^u  f := f \circ S_t^u , \quad f \in \mathcal{F},
\end{equation}
and its generator by 
\begin{equation}
    \label{eq:paraGeneratordef}
    L_F^u f (x) := {\rm D}_x f (x) \cdot F^u (x), 
\end{equation}
where $\mathcal{F}=C^k(\mathbb{X})$ with $1\leq k\leq r$. 
The following theorem clarifies the strong convergence of the %parameterized Koopman 
semigroup 
$\{U_t^u\}_{t\geq 0}$ of the parameterized Koopman operators acting on $C^k (\mathbb{X})$. 
\begin{theorem}
\label{thm:scKO}
For any $t\geq 0$, $U_t^u$ strongly converges to $U_t^{u_0}$ as $u \to u_0$ with $u_0 \in \mathbb{U}$ in the $C^k(\mathbb{X})$ norm for all $k \leq r$. 
\end{theorem}
{\color{black}\begin{proof}
After discretizing the flow $S_t^u$ at $t$, then the idea of \cite{irwin1972smoothness} can be applied. 
\end{proof}}

\subsection{Parameter-Dependence of Koopman Eigenfunctions}
Now, we analyze the parameter-dependence of the Koopman eigenfunctions. 
We suppose the case where the system \eqref{eq:paraDS} has a GES EP, namely:
\begin{assumption}
\label{asm:GES_FP}
The system \eqref{eq:paraDS} has a GES EP at the origin
%\footnote{For simplicity, we fix the parameterized GES EP at the origin. If one wishes to consider the case where $\dot{x}= F^u (x)$ has a GES EP $x_0^u$, the following results can still be applied by rewriting the system as $\dot{x} = F^u  (x + x_0^u)$. } 
$\{0\} \in {\rm int}(\mathbb{X})$ for all $u  \in \mathbb{U}$. 
\end{assumption}
{\color{black}
\begin{remark}
\label{rem:strict_assumptions}
Throughout the remainder of this section we assume that the state space $\mathbb{X}$ is compact and independent of $u$ (Assumption \ref{asm:state_space}), and the GES EP is fixed at the origin (Assumption \ref{asm:GES_FP}). 
However, these assumptions are not essential and can be relaxed as follows:
\begin{enumerate}[label=A\arabic*)]
    \item The state space $\mathbb{X}^u$ is the closure of a precompact open set and is positively invariant for all $u\in \mathbb{U}$, although it may depend on $u$.
    \item The system \eqref{eq:paraDS} has a GES EP for all $u\in \mathbb{U}$
    \item There exists a $C^r$ fiber bundle structure\cite{Lee_smooth} $\pi:\mathbb{E} \to \mathbb{U}$, where $\mathbb{E}:=\bigsqcup_{u \in \mathbb{U}} \mathbb{X}^u \times \{u\}$.
\end{enumerate}
According to \cite{eldering2018global}, assumptions A1) and A2) yield the existence of a fiber bundle structure $\pi:\mathbb{E} \to \mathbb{U}$. 
Thus, assumption A3) further requires this fiber bundle structure to be $C^r$. 
We do not elaborate on this generalization here.
The key point is that the $C^r$ fiber bundle structure gives a local trivialization over an open cover $\{\mathbb{V}_{\alpha}\}$ of $\mathbb{U}$. 
In particular, for any $u_1,u_2\in\mathbb{V}_{\alpha}$, one obtains a diffeomorphism $T_{u_1,u_2}:\mathbb{X}^{u_1}\to\mathbb{X}^{u_2}$ depending smoothly on $(u_1,u_2)$ and fixing the GES equilibrium point. 
Hence, the fibers can be identified locally with a common reference space.
\end{remark}
}
Under Assumption \ref{asm:GES_FP}, the Jacobian ${\rm D}_x F^u  (0)$ has $n$ eigenvalues $\lambda_{1}^u , \ldots, \lambda_{n}^u $ depending on $u$ whose absolute values are strictly smaller than $1$. 
\begin{remark}
    \label{rem:eigenvalue_continuous}
    According to \cite[Chapter~2.5]{kato2013perturbation}, each eigenvalue of the parameterized matrix $A^u ={\rm D}_x F^u  (0)$ is continuous in $u$, implying the continuity of the Koopman principal eigenvalues. 
\end{remark}
The following theorem shows the strong convergence of the resolvent of the parameterized Koopman operator. 
\begin{theorem}
\label{thm:resolvent_KO_para}
Consider the system \eqref{eq:paraDS} with Assumptions \ref{asm:state_space}, \ref{asm:continuity_para}, and \ref{asm:GES_FP} and let $1 \leq k\leq r$ and $u _0 \in \mathbb{U}$ be arbitrary. 
Then, for all $\zeta \in \rho(L_F^{u_0})$, % with ${\rm Re} \, \zeta > k\cdot \max_{1\leq \ell \leq n} {\rm Re}\, \lambda_\ell^u$, 
$R(\zeta,L_F^u )$ strongly converges to $R(\zeta,L_F^{u_0})$. 
\end{theorem}
{\color{black}\begin{proof}
See Appendix~\ref{app:1}. 
\end{proof}}

For the principal eigenvalues $\lambda_1^u,\ldots, \lambda_n^u$, we can define the associated eigenprojections by \eqref{eq:integral_resolvent}, denoted by $P_{\lambda_1^u}^u,\ldots, P_{\lambda_n^u}^u$. 
We can then present the strong convergences of the eigenprojections from Theorem~\ref{thm:resolvent_KO_para}, Proposition~\ref{cor:strong_convergence_P}, Remark~\ref{rem:eigenvalue_continuous}, and Proposition~\ref{prop:kval}. 
\begin{theorem}
\label{thm:convergence_KEP_nonresonant}
Consider the system \eqref{eq:paraDS} with Assumptions \ref{asm:state_space}, \ref{asm:continuity_para}, and \ref{asm:GES_FP} and let $1\leq k\leq r$ and $u _0 \in \mathbb{U}$ be arbitrary. 
If $\lambda_i^{u_0}$ with $i \in \{1,\ldots, n\}$ satisfies the $k$-nonresonant and $k$-spectral spread conditions, the eigenprojection $P_{\lambda_i^{u}}^u$ associated with $\lambda_i^{u}$ strongly converges to $P_{\lambda_i^{u_0}}^{u_0}$ as $u \to u_0$. 
\end{theorem}
\begin{corollary}
    \label{cor:KEF}
Let $1\leq k\leq r$ and $u _0 \in \mathbb{U}$ be arbitrary. 
If $\lambda_i^{u_0}$ satisfies the $k$-nonresonant and $k$-spectral spread conditions, there exist an open set
\footnote{\color{black}
This set $\mathbb{V}$ is chosen so that $\lambda_i^{u}$ satisfies the $k$-nonresonance and $k$-spectral spread conditions for all $u\in\mathbb{V}$, and so that the vector bundle with fiber ${\rm span}\{\phi_{\lambda_i^u}^u\}$ at each $u\in\mathbb{V}$ is trivial.
}
$\mathbb{V}\subset \mathbb{U}$ containing $u_0$ and a function $\phi_{\lambda_i^{(\cdot)}}^{(\cdot)}: \mathbb{V} \to C^k (\mathbb{X})$ such that this function is continuous in the $C^k(\mathbb{X})$ norm and that $\phi_{\lambda_i^{u}}^u$ is the Koopman principal eigenfunction associated with $\lambda_i^u$. 
\end{corollary}

\subsection{Global Linearization Result} %with a Parameter 

As it is shown that the Koopman principal eigenfunctions provide a diffeomorphism linearizing nonlinear systems \cite{kvalheim2021existence}, it can also be shown that there exists a \emph{parameterized} diffeomorphism linearizing the \emph{parameterized} nonlinear system \eqref{eq:paraDS}. 
\begin{theorem}
\label{thm:linearization}
Consider the system \eqref{eq:paraDS} with Assumptions \ref{asm:state_space}, \ref{asm:continuity_para}, and \ref{asm:GES_FP} and let $1\leq k\leq r$ be arbitrary. 
Assume that for all $u  \in \mathbb{U}$, the eigenvalues of the Jacobian, denoted by $\lambda_1^u ,\ldots, \lambda_n^u $, satisfy the $k$-nonresonant and $k$-spectral spread conditions. 
Then, there exists a diffeomorphism $\psi^u  = (\psi_1^u ,\ldots,\psi_n^u )^\top:\mathbb{X}\to \psi^u  (\mathbb{X})\subset \mathbb{R}^n$ such that ${\rm D}_x \psi^u (0) = I$ and
\begin{equation}
\label{eq:linear_conjugacy}
\frac{d}{dt}\psi^u  (x) = A^u \psi^u (x), \quad \forall x \in \mathbb{X} ,~ \forall t > 0
\end{equation}
holds, where $A^u = {\rm D}_x F^u (0)$. 
Moreover, $\psi_i^u$ is continuous as a function of $u$ in the $C^k (\mathbb{X})$ norm. 
\end{theorem}
{\color{black}\begin{proof}
    See Appendix~\ref{app:1}. 
\end{proof}}

By the transformation $z = \psi^u (x)$, we obtain
\begin{equation}
\label{eq:linearized_z}
    \dot{z} = A^u z,\quad \forall z\in \mathbb{Z}^u , ~ \forall u \in \mathbb{U}, 
\end{equation}
where $\mathbb{Z}^u := \psi^u (\mathbb{X})$. 
This results in the existence of a finite-dimensional representation of the parameterized %a 
nonlinear %dynamical 
system \eqref{eq:paraDS}. 
The representation %with a parameter, which 
is linear in terms of the state but nonlinear in terms of the parameter. 
Furthermore, from the uniqueness of the Koopman eigenfunctions, if two of $\psi_1^u$ and $\psi_2^u$ satisfy \eqref{eq:linear_conjugacy} and ${\rm D}_x \psi_1^u (0) = {\rm D}_x \psi_2^u (0) = I$ for all $u \in\mathbb{U}$, then it follows that $\psi_1^u - \psi_2^u = 0$ for all $u \in \mathbb{U}$, indicating the uniqueness of the Sternberg linearization \cite{sternberg1957local,kvalheim2021existence} with the parameter. 
We also note that for any invertible matrix
\footnote{\color{black}$G^u$ can not be chosen as a matrix diagonalizing ${\rm D}_x F^u(0)$ in general since the vector bundle constructed by attaching the eigenspace of ${\rm D}_x F^u(0)$ to $u \in \mathbb{U}$ can be nontrivial. }
$G^u \in \mathbb{R}^{n\times n}$ which is continuous in $u$, 
$\Psi^u (x) = G^u \psi^u (x)$ also satisfies \eqref{eq:linear_conjugacy} with $A^u = G^u {\rm D}_x F^u(0) (G^u)^{-1}$.

\subsubsection*{Related Work}
We note the connection between Theorem~\ref{thm:linearization} and the Koopman operators for nonlinear %control 
systems with inputs. 
One of the direct and simple formulations of the Koopman operator for a system with input %control 
is to consider an augmented system $\dot{x} = F^u(x)$ and $\dot{u} = 0$, which simplifies the implementation of EDMD with control \cite{proctor2018generalizing,shi2022deep}. 
Our result implies the existence of a continuous map $(x,u) \mapsto \Psi^u (x)$, enabling the finite-dimensional linearization of the system with input (although it will be explained in Section~\ref{sec:4} that this linearization is not suitable for control). 
The continuity result %is finding 
justifies the use of universal approximation theorems {\color{black}(see, e.g., \cite{cybenko1989approximation,lusch2018deep})}, which approximate $\psi^u(x)$ by a large number of continuous functions, and thus guarantees %ensures 
the accuracy and convergence of EDMD. 
In addition, our result theoretically ensures a control strategy for fast convergence towards the GES EP as shown in \cite{banks2023koopman}.

%%%%%
\section{Linearization of Control-Affine Systems}
\label{sec:4}
\subsection{Motivation and Related Work} %Background}
As stated in Section~\ref{sec:1}, we have been motivated by the global linearization of nonlinear systems with inputs. 
Here, let us consider a control-affine system with a time-dependent input $u(t)$ as
\begin{equation}
    \label{eq:affine_eq}
    \dot{x} (t) = F(x(t)) + \sum_{i=1}^d u_i(t) G_i (x(t)),\quad x(t) \in \mathbb{X},
\end{equation}
where $u(t)= (u_1(t),\ldots, u_d(t))^\top \in \mathbb{U} \subset \mathbb{R}^d$. 
While Section~\ref{sec:3} %the previous section 
clarified the existence of a map $\psi^{u(t)}$ that linearizes the parameterized system, % $x(t)$, 
this $\psi^{u(t)}$ depends on $u(t)$, which yields the state equation in terms of the transformed state $z(t) = \psi^{u(t)} (x(t))$, given by 
$$
\dot{z}(t) = A^{u(t)} z(t) + {\rm D}_u \psi^{u(t)} \big( (\psi^{u(t)})^{-1} (z(t)) \big) \dot{u}(t),
$$
where $u(t)$ is assumed to be differentiable in $t$. 
This indicates that the transformation of Theorem~\ref{thm:linearization} is not suitable for control unless $\psi^u(t)$ is independent of $u(t)$. 
This motivates to find %gives us the motivation to find 
a condition such that there exists a linearizing map $\psi$ independent of $u$. 

In previous research, the linearization of the control-affine system %nonlinear control systems 
\eqref{eq:affine_eq} has been studied with %based on 
differential geometry \cite{khalil2002nonlinear} and Lie algebra \cite{elliott2009bilinear}. 
For the case where $F$ and $G_i$ have a common EP%equilibrium point
\footnote{In such a case, geometric control theory yields impossibility of feedback linearization \cite{khalil2002nonlinear}.} (we let be at $\{0\}$ for simplicity), i.e., $F(0)=G_i(0)=0$, \cite{elliott2009bilinear} showed a sufficient condition of local bilinearizability as follows: 
\begin{proposition}[Theorem 7.8 in \cite{elliott2009bilinear}]
\label{prop:lobal_bilinearization}
Assume that {\color{black}$F, G_i \in C^\omega(\mathbb{X},\mathbb{R}^n)$ (analytic vector fields)} and $F(0)=G_i(0) = 0$. 
Let $A:= {\rm D}_x F(0)$, $B_i := {\rm D}_x G_i (0)$, and assume that the eigenvalues of $A$ (possibly unstable) satisfies the $\infty$-nonresonant condition. 
Denote the Lie algebra generated by $F, G_1, \ldots, G_d$ as $\{F,G_1,\ldots, G_d \}_{\mathfrak{L}}$ and the Lie algebra generated by $A, B_1, \ldots, B_d$ as $\{A,B_1,\ldots, B_d \}_{\mathfrak{L}}$. 
If there exists an isomorphism between $\{F,G_1,\ldots, G_d \}_{\mathfrak{L}}$ and $\{A,B_1,\ldots, B_d \}_{\mathfrak{L}}$ in the sense of Lie algebra, then there exist a neighborhood $V$ of $\{0\} \in \mathbb{X}$ and a $C^\omega$ diffeomorphism $\psi:V\to \mathbb{R}^n$ such that 
$$
\frac{d}{dt} \psi(x) = A \psi(x) + \sum_{i=1}^d u_i B_i \psi(x),\quad x \in \mathbb{X}
$$
holds. 
\end{proposition}
Here, recall that the Lie algebra has a binary operation denoted by $[\cdot,\cdot]$, called Lie bracket, satisfying Jacobi's axioms. 
This bracket is defined by $[A,B] := AB - BA$ for the case that the Lie algebra consists of matrices, whereas $[F,G] f := L_F L_G f - L_G L_F f$ for all $f \in C^r (\mathbb{X})$ for the case that the Lie algebra consists of vector fields. 
Recall also that an isomorphism $\pi:\mathfrak{g} \to \mathfrak{g}'$ of two Lie algebras $\mathfrak{g}$ and $\mathfrak{g}'$ satisfies (i) $\pi (\alpha A + \beta B) = \alpha \pi(A) + \beta \pi(B)$ for all $A, B \in \mathfrak{g}$ and all $\alpha, \beta \in \mathbb{C}$; (ii) $\pi([A,B]) = [\pi(A) ,\pi (B)]$ for all $A, B \in \mathfrak{g}$.

\subsection{Global Bilinearization Result}
Here, we extend the local statement of Proposition~\ref{prop:lobal_bilinearization} to a global one using the Koopman operator. 
We here make the following assumptions: %assume the followings:
\begin{assumption}
\label{asm:control1}
The dynamical system $(\mathbb{X},F)$ is positively invariant and has a GES EP at the origin. 
\end{assumption}
\begin{assumption}
\label{asm:control2}
The vector fields $F,G_1,\ldots,G_d$ are $C^r$ and $G_i(0) = 0$ for all $i\in \{1,\ldots,d\}$. 
\end{assumption}
\begin{theorem}
\label{thm:bilinearization}
For the control-affine system $\dot{x}=F(x)+\sum_{i=1}^{d}uG_i(x)$,  %\eqref{eq:affine_eq}, 
assume Assumptions \ref{asm:control1}, \ref{asm:control2}.
Let $A:= {\rm D}_x F(0)$, $B_i := {\rm D}_x G_i (0)$, and assume that the eigenvalues of $A$ satisfies the $k$-nonresonant and $k$-spectral spread conditions, where $2\leq k \leq r$. 
Denote the Lie algebra generated by $F, G_1, \ldots, G_d$ as $\{F,G_1,\ldots, G_d \}_{\mathfrak{L}}$ and the Lie algebra generated by $A, B_1, \ldots, B_d$ as $\{A,B_1,\ldots, B_d \}_{\mathfrak{L}}$. 
If there exists an isomorphism between $\{F,G_1,\ldots, G_d \}_{\mathfrak{L}}$ and $\{A,B_1,\ldots, B_d \}_{\mathfrak{L}}$, 
we have a diffeomorphism $\psi:\mathbb{X} \to \mathbb{R}^n$ such that $z=\psi(x)$ satisfies 
\begin{equation}
\label{eq:bilinearization}
    \dot{z} = Az + \sum_{i=1}^d u_i B_i z,\quad x \in \mathbb{X}. 
\end{equation}
\end{theorem}
{\color{black}\begin{proof}
    See Appendix~\ref{app:1}. 
\end{proof}}

%Note that 
Theorem~\ref{thm:bilinearization} also shows the sufficient condition under which the parameterized Koopman eigenfunctions introduced in Section~\ref{sec:3} are independent of $u$. 
Specifically, letting $F^u := F+\sum_{i=1}^d u_i G_i$ and $\lambda_1^u,\ldots,\lambda_n^u$ be the eigenvalues of ${\rm D}_x F^u (0)$, we have
\begin{equation}
\nonumber
L_F^u \phi_{\lambda_i} = \lambda_i^u \phi_{\lambda_i},\quad \forall i\in \{1,\ldots n\},~\forall u\in \mathbb{U},
\end{equation}
where $\phi_{\lambda_1},\ldots,\phi_{\lambda_n}$ are the Koopman eigenfunctions of $L_F$. 

\subsubsection*{Difference from Proposition~\ref{prop:lobal_bilinearization}}
Since the construction of Proposition~\ref{prop:lobal_bilinearization} is based on the Taylor expansion of the analytic vector fields, the bilinearization is valid only in the neighborhood where the Taylor expansion converges. 
On the other hand, since the construction of Theorem~\ref{thm:bilinearization} is based on the spectral property of the Koopman generator, we obtain the global bilinearization. 
Furthermore, we utilized the algebraic structure of the Koopman generator in this proof. 
We speculate that analyzing the algebraic structure of $L_F,L_{G_1},\ldots, L_{G_d}$ can extend the result on bilinearization to other classes of control systems. 

{\color{black}
\begin{remark}
\label{rem:feedback}
Although Theorem \ref{thm:bilinearization} assumes that the drift vector field $F$ in \eqref{eq:affine_eq} has a GES EP at the origin, this assumption can be relaxed by means of a feedback transformation of the form $u_i=\alpha_i(x)+\sum_{j=1}^d \beta_{ji}(x)v_j$.
More precisely, suppose that there exists a feedback law $u=(\alpha_1(x),\ldots,\alpha_d(x))^\top$ that renders the origin exponentially stable. 
Now consider the input transformation
$u_i=\alpha_i(x)+\sum_{j=1}^d \beta_{ji}(x)v_j$.
Then Theorem \ref{thm:bilinearization} suggests that a feedback bilinearization
$\dot{z}=A'z+\sum_{i=1}^d v_i B_i' z$
can be achieved provided that one can choose $\beta_{ji}:\mathbb{X}\to\mathbb{R}$ so that the Lie algebra
$$\textstyle
\left\{ F+\sum_{i=1}^d \alpha_i G_i, ~\sum_{j=1}^d \beta_{j1}G_j, \ldots, ~\sum_{j=1}^d \beta_{jd}G_j \right\}_{\mathfrak L}
$$
is isomorphic to $\{A',B_1',\ldots,B_d'\}_{\mathfrak L}$, where 
$A' := D_x (F+\sum_{i=1}^d \alpha_i G_i)(0)$ and $B_i':=D_x(\sum_{j=1}^d \beta_{ji} G_j)(0)$.
\end{remark}

As indicated in Remark \ref{rem:feedback}, we believe that our bilinearization result can be extended to a class of systems for which the origin is globally exponentially stabilizable. This would allow one, for example, to apply optimal control techniques for bilinear systems.
}

\subsubsection*{Related Work}
Bilinearization based on the Koopman operator has been extensively studied{\color{black}: see, e.g., \cite{goswami2021bilinearization,peitz2020data,bruder2021advantages}}. % a lot. 
Particularly, EDMD with control has been studied to find a map $\psi:\mathbb{X} \to \mathbb{C}^N$ into the form 
$$\dot{z}=Az + \sum_{i=1}^d u_i B_i z ,$$
where $\psi$ is independent of $u$. 
The validity of this approach was analyzed in \cite{goswami2021bilinearization}, which showed that a sufficient condition for the existence of a finite-dimensional bilinear representation is the existence of a so-called Koopman invariant subspace $\mathcal{S}\subset \mathcal{F}$, which is invariant under all the actions of $L_F, L_{G_1},\ldots,L_{G_d}$. 
The Lie-algebraic condition in Theorem~\ref{thm:bilinearization} is relatively constructive since it is expressed directly for %in terms of 
the underlying vector fields. 
We also speculate that the condition is intrinsic to the control-affine system %more fundamental 
since it preserves the Lie-algebraic structure studied in the traditional nonlinear control theory \cite{elliott2009bilinear}.
It is also worth noting that the Lie-algebraic condition in Theorem~\ref{thm:bilinearization} can be explicitly and directly verified for given vector fields, whereas verifying the existence of a Koopman invariant subspace is a challenging issue. 

{\color{black}
\section{An Example}
Consider the following two-dimensional dynamical system with a parameter (or input) $u$ and a parameter $a\in \mathbb{R}$:
\begin{equation}
\label{eq:ex1}
\left\{
\begin{aligned}
    \dot{x}_1 &= -x_1,\\
    \dot{x}_2 &= (-1+u)x_2 + (a+u)x_1^2,
\end{aligned}
\right.
\quad
\begin{aligned}
    (x_1,x_2) &\in \mathbb{X}\subset \mathbb{R}^2,\\
    u &\in (-\infty,1).
\end{aligned}
\end{equation}
Here, $\mathbb{X}$ is chosen so that $0\in\mathbb{X}$ and $\mathbb{X}$ is positively invariant.
We define the vector fields $F(x_1,x_2):= (-x_1,-x_2+ax_1^2)^\top$ and $G(x_1,x_2):= (0,x_2+x_1^2)^\top$
so that \eqref{eq:ex1} can be written as
$\dot{x} = F(x)+uG(x)$. 

When $u$ is constant, the origin is GES EP, and the Jacobian matrix at the origin has eigenvalues $-1$ and $-1+u$.
In this case, the $\infty$-nonresonance
\footnote{Since the system is $C^\infty$, $k$ in Definition \ref{def:k-nonresonant} can be taken as $\infty$. Moreover, when $k=\infty$, Definition \ref{def:spectral_spread} implies that the spectral spread condition is unnecessary.}
condition requires that there exist no $m_1, ~m_2\in\mathbb{N}_{\geq 1}$ such that $-1 = m_2(-1+u)$ and $-1+u = m_1(-1)$, 
equivalently,
{
\setlength{\abovedisplayskip}{4pt}
\setlength{\belowdisplayskip}{4pt}
$$
u \in (-\infty,1) \setminus \left( \left\{0,1/2,2/3,3/4,\ldots \right\} \cup \left\{0,-1,-2,\ldots\right\} \right).
$$
}%

According to Corollary \ref{cor:KEF}, there exist two principal eigenfunctions as long as the $\infty$-non\-resonance condition is satisfied.
Indeed, one can verify that
$$\phi_{-1}(x_1,x_2):=x_1,\quad\phi_{-1+u}(x_1,x_2):=
x_2+\frac{a+u}{1+u}x_1^2$$
satisfy $L_{F+uG}\phi_{-1} = -\phi_{-1}$ and $L_{F+uG}\phi_{-1+u} = (-1+u)\phi_{-1+u}$, 
provided that $u\neq -1$.
Hence, these functions are Koopman eigenfunctions, and they depend continuously on $u$ in the $C^k$ norm for any $k \in \mathbb{N}_{\geq 1}$.
When $u=-1$, the function $\phi_{-1+u}$ is not well defined because the denominator $1+u$ vanishes; this is consistent with the resonance at $u=-1$.
Moreover, the diffeomorphism $\psi^u:\mathbb{X}\to\mathbb{R}^2$ defined by
$\psi^u(x):= (\phi_{-1}(x),\phi_{-1+u}(x))^\top$
gives the coordinate transformation $z=\psi^u(x)$ which yields the linearized system
{
\setlength{\abovedisplayskip}{4pt}
\setlength{\belowdisplayskip}{4pt}
$$
\dot{z}= \begin{pmatrix}
-1 & 0 \\ 0 & -1+u
\end{pmatrix}z.
$$
}%
This is also consistent with Theorem \ref{thm:linearization}. 

Now, we regard $u$ as an input. 
A direct calculation gives $[F,G] = (0,(a-1)x_1^2)^\top$ and $[D_xF(0),D_xG(0)]=0$. 
Hence, when $a=1$, we have $[F,G]=0$, and therefore the condition in Theorem \ref{thm:bilinearization} is satisfied.
In this case, the diffeomorphism $\psi(x_1,x_2)= (x_1,x_2+x_1^2)^\top$ yields the bilinearized system
{
\setlength{\abovedisplayskip}{4pt}
\setlength{\belowdisplayskip}{4pt}
$$
\dot{z} = \begin{pmatrix}
    -1 & 0 \\ 0 & -1
\end{pmatrix} z + u(t) \begin{pmatrix}
    0 & 0 \\ 0 & 1
\end{pmatrix} z, 
$$
}%
where $u(t)$ can depend on $t$. 
}

%%%%%
\section{Conclusion}
\label{sec:5}
In this paper, we introduced the Koopman operator for a parameterized dynamical system with %that has 
a globally exponentially stable equilibrium point, and we analyzed the parameter-dependence of its eigenfunctions. 
As a main result, we obtained a global linearization that depends continuously on the parameter. 
We then investigated the conditions under which the transformation that provides the linearization becomes independent of the parameter for control-affine systems. 
This results in a global bilinearization that is independent of the parameter. 
Finally, since we focused on the system with GES EP in this paper, which is somewhat restricted, %Focusing on systems with a GES EP is somewhat conservative, and 
one of our future directions is to extend the present results to broader classes of systems.

\section*{Acknowledgement}
The authors thank the anonymous reviewers for their careful reading and valuable comments and suggestions.
The authors used ChatGPT (OpenAI) to assist in improving the clarity and readability of the English text in the Introduction section of this paper. All AI-generated suggestions were reviewed and edited by the authors to ensure accuracy and integrity. 

\appendices
\section{Collection of Proofs}
\label{app:1}
\subsection{Proof of Proposition \ref{prop:resolvent_KO}}
\begin{proof}
Here, denote $\omega := k\cdot {\rm max}_{1\leq \ell \leq n} {\rm Re} \, \lambda_\ell$. 
The assumption (i) ensures the existence of $\epsilon>0$ such that ${\rm Re}\,\zeta - \epsilon > \omega$. 
Let $\tilde{\mathcal{F}}$ be a closed subspace of $C^k(\mathbb{X})$ given by
\begin{equation}
\label{eq:closed_subspace}
\tilde{\mathcal{F}} := \{ f\in C^k(\mathbb{X}) ~|~ \forall i \in \{1,\ldots,k\} ,~{\rm D}_x^i f (0) = 0 \}. 
\end{equation}
Then, according to \cite[Lemmas 5]{kvalheim2021existence}, it can be shown that there exists an adopted small ball $B \subset \mathbb{X}$ containing $\{0 \}$ such that this ball is strictly positively invariant and
$$
\begin{aligned}
&\exists M > 0 \quad {\rm such ~ that} \\ &\forall f|_B \in \tilde{\mathcal{F}}|_{B}, ~\|{\rm e}^{-(\zeta - \epsilon) t} U_t f|_B \|_{C^k(B)} \leq M \|f|_B \|_{C^k(B)}
\end{aligned}
$$
holds\footnote{While \cite{kvalheim2021existence} showed that ${\rm e}^{-\omega t} U_t$ is a strictly contraction map with any $t>0$ under the Banach space $\tilde{\mathcal{F}}|_{B}$, we obtain from this idea that $\{ {\rm e}^{-\omega t} U_t \}_{t\geq 0}$ is a contraction semigroup. }, 
where $\tilde{\mathcal{F}}|_{B} := \{ f|_B \in C^k(B)~|~ f \in  \tilde{\mathcal{F}} \}$. 
According to Hille-Yosida theorem (see, e.g., \cite{kato2013perturbation}), the generator of ${\rm e}^{-(\zeta - \epsilon) t} U_t|_B$, which is given by $(L_F - (\zeta - \epsilon) I)|_B$, has a property that any $\theta > 0$ belongs to the resolvent set of $(L_F - (\zeta - \epsilon) I)|_B$, implying that $(L_F - (\zeta - \epsilon) I - \theta I)^{-1}|_B$ exists and is bounded. 
By choosing $\theta = \epsilon$, we obtain that $(L_F - \zeta I)^{-1}|_B$ exists and is bounded. 

Next, consider the existence and uniqueness of the solution $f \in C^k(\mathbb{X})$ of the linear equation 
\begin{equation}
    \label{eq:inverse_resolvent}
    (\zeta I - L_F) f = g, 
\end{equation}
for any given $g \in C^k (\mathbb{X})$. 
If the existence and uniqueness are proven, then $(\zeta I - L_F)$ has inverse with domain $C^k(\mathbb{X})$, implying $\zeta \in \rho({\color{black} L_F})$ \cite[Chap. 3, Problem 6.1]{kato2013perturbation}. 
According to \cite[Lemmas 1 and 4]{kvalheim2021existence}, the assumption (ii) implies that there uniquely exist a $k$-th order polynomial $p$ and $r\in \tilde{\mathcal{F}}$ such that 
\begin{equation}
\label{eq:polynomial}
\zeta p(x) -  {\rm D}_x p (x) F(x) = g(x) + r(x)
\end{equation}
holds. 
The operation $r|_B \mapsto (L_F - \zeta I)^{-1}|_B r|_B$ is bounded according to the previous consideration. 
Since the system contains the GES EP and $\mathbb{X}$ is compact, there exists $T>0$ such that for any $x\in \mathbb{X}$, $S_T(x) \in B$. 
Here, define $V_T:C^k(B) \to C^k(\mathbb{X})$ by 
$$
(V_T f|_B) (x) :=  {\rm e}^{-\zeta T} f|_B (S_T(x)), ~~\forall f|_B \in C^k(B), ~ \forall x \in \mathbb{X}. 
$$
This operator is well-defined and is bounded according to \cite{irwin1972smoothness}. 
Therefore, the operation $r|_B \mapsto V_T (L_F - \zeta I)^{-1}|_B r|_B$ is well-defined. 
Furthermore, it can be seen that $(L_F - \zeta I) V_T f|_B = V_T (L_F - \zeta I)|_B f|_B$ holds for all $f|_B \in C^k(B)$. 
Now, we define
$$
f = p + V_T (L_F - \zeta I)^{-1}|_B \, r|_B - \int_0^T {\rm e}^{-\zeta t} U_t r dt .
$$
Then, we have
$$
(\zeta I - L_F) p = g+r,\qquad \qquad \qquad \qquad \qquad \qquad \quad ~~ 
$$
$$
\begin{aligned}
&(\zeta I - L_F ) V_T (L_F - \zeta I)^{-1}|_B \, r|_B \\ &\qquad = V_T (\zeta I - L_F )|_B (L_F - \zeta I)^{-1}|_B \, r|_B = - V_T r|_B, 
\end{aligned}
$$
and
$$
\begin{aligned}
&(\zeta I - L_F) \int_0^T {\rm e}^{-\zeta t} U_t r dt = \int_0^T (\zeta I - L_F) {\rm e}^{-\zeta t}  U_t r dt \\
&\qquad = \int_0^T \frac{\partial}{\partial t} \left( -{\rm e}^{-\zeta t} U_t r \right) dt 
= r - {\rm e}^{-\zeta T} U_T r \\
&\qquad = r - V_T r|_B, 
\end{aligned}
$$
which yields
$$
(\zeta I - L_F ) f = g.
$$
This implies that $f$ is the solution of \eqref{eq:inverse_resolvent}, completing the proof. 
\end{proof}

\subsection{Proof of Theorem \ref{thm:resolvent_KO_para}}
\begin{proof}
First, Proposition \ref{prop:resolvent_KO} and the continuity of the Koopman eigenvalues (Remark \ref{rem:eigenvalue_continuous}) imply $\zeta \in \rho(L_{F}^u)$ for all $u \in \bar{B}^{\mathbb{R}^d}(u_0,\delta)$ with sufficiently small $\delta>0$. 

Recall the proof of Proposition \ref{prop:resolvent_KO} and consider the solution of the linear equation
\begin{equation}
    \label{eq:inverse_resolvent_para}
    (\zeta I - L_F^u ) f^u = g, 
\end{equation}
for any given $g \in C^k (\mathbb{X})$.  
As shown in the proof of Proposition \ref{prop:resolvent_KO}, there uniquely exists a $k$-th order polynomial $p^u$ such that
\begin{equation}
\label{eq:polynomial_para}
\zeta -  {\rm D}_x p^u (x) F(x) = g(x) + r^u(x) 
\end{equation}
holds, where $r^u := g - (\zeta I - L_F^u) p^u$ is the remainder. 
This $p^u$ linearly depends on $g$ and is determined by the Taylor coefficients of $g$ up to $k$-th order \cite[Lemma 4]{kvalheim2021existence}. 
Then, according to the perturbation theory for finite-dimensional vector space, the $k$-nonresonant condition ensures the strong convergence $p^u \to p^{u_0}$ in the $C^k(\mathbb{X})$ norm, yielding the strong convergence $r^u \to r^{u_0}$ as well. 

According to the proof of Proposition \ref{prop:resolvent_KO}, 
$$
f^{u_0} = p^{u_0} + V_T^{u_0} (L_F^{u_0} - \zeta I)^{-1}|_B \, r^{u_0}|_B - \int_0^T {\rm e}^{-\zeta T} U_T^{u_0} r^{u_0} dt 
$$
is the solution of \eqref{eq:inverse_resolvent_para}, where $V_T^{u_0} := {\rm e}^{-\zeta T} U^{u_0}_T$. 
By taking $\delta>0$ sufficiently small, for all $u \in \bar{B}^{\mathbb{R}^d} (u_0, \delta)$, $B$ remains positively invariant and $S_T^u(\mathbb{X}) \subset B$ follows. 
This fact and Theorem \ref{thm:scKO} imply that $V_T^{u}$ and $\int_0^T {\rm e}^{-\zeta T} U_T^u dt$ strongly converges as $u \to u_0$. 
Therefore, if the strong convergence of $(L_F^{u} - \zeta I)^{-1}|_B$ is shown, we obtain the strong convergence of
$$
f^{u} = p^{u} + V_T^{u} (L_F^{u} - \zeta I)^{-1}|_B \, r^{u}|_B - \int_0^T {\rm e}^{-\zeta T} U_T^{u} r^{u} dt 
$$
to $f^{u_0}$, implying the strong convergence of $R(\zeta, L_F^u)$. 

Then we show the strong convergence of $(L_F^{u} - \zeta I)^{-1}|_B$. 
By taking $\delta>0$ sufficiently small, for all $u \in \bar{B}^{\mathbb{R}^d} (u_0, \delta)$, $\{ {\rm e}^{-\omega^u t} U_t^u \}_{t\geq 0}$ is a contraction semigroup, where $\omega^u := k\cdot \max_{1\leq \ell \leq n} {\rm Re}\, \lambda_\ell^u$. 
Furthermore, Theorem \ref{thm:scKO} implies the strong convergence of ${\rm e}^{-\zeta t} U_t^u$ to ${\rm e}^{-\zeta t} U_t^{u_0}$ for any $t\geq 0$. 
Therefore, according to \cite[Theorem 2.16, Chap. 9]{kato2013perturbation}, $(L_F^{u} - \zeta I)^{-1}|_B$ strongly converges to $(L_F^{u_0} - \zeta I)^{-1}|_B$. 
\end{proof}

\subsection{Proof of Theorem \ref{thm:linearization}}
Before the proof of Theorem \ref{thm:linearization}, we have the following Lemma.
\begin{lemma}
\label{lem:projection_differential}
Consider the system \eqref{eq:DS_nopara} with a GES EP and let $L_F$ be the Koopman generator defined by \eqref{eq:generator} with $\mathcal{F} = C^k(\mathbb{X})$. 
Let $\lambda_i$ be the Koopman principal eigenvalues with $i\in \{ 1,\ldots,n\}$ and $P_{\lambda_i}$ be the associated eigenprojection. 
Additionally, let $Q_i^*$ be the dual operator of an eigenprojection of the Jacobian ${\rm D}_x F (0)$ associated with $\lambda_i$. 
If $\lambda_i$ satisfies $k$-nonresonant and $k$-spectral spread condition, then 
\begin{equation}
\label{eq:projection_differential}
L_D (P_{\lambda_i} g) = Q_i^* {\rm D}_x g(0) ,\quad g \in C^k(\mathbb{X})
\end{equation}
holds, where $L_D:{C}^k (\mathbb{X}) \to (\mathbb{C}^n)^*$ is a linear bounded operator defined by $L_D g := {\rm D}_x g(0)$. 
\end{lemma}
{\color{black}
\begin{proof}
We first study $R(\zeta,L_F) g = (L_F - \zeta I)^{-1}g$ for any given $g\in C^k (\mathbb{X})$ and any $\zeta \in \rho (L_F)$. 
By virtue of \eqref{eq:polynomial} in the proof of Proposition \ref{prop:resolvent_KO}, there exist a $k$-th order polynomial $p_\zeta$ and a residual $r_\zeta \in \tilde{F}$ (recall the definition of $\tilde{F}$ from \eqref{eq:closed_subspace}) such that $g$ can be represented by 
\begin{equation}
    \label{eq:lemma}
g = (\zeta I - L_F) p_\zeta - r_\zeta, 
\end{equation}
where both $p_\zeta$ and $r_\zeta$ depend on $\zeta$. 
Furthermore, according to the proof of Proposition \ref{prop:resolvent_KO}, we have
$$
(\zeta I - L_F)^{-1} g = p_\zeta + V_T (L_F - \zeta I)^{-1}|_B r_\zeta|_B - \int_0^T {\rm e}^{-\zeta t} U_t r_\zeta dt,
$$
where $V_T$ and $B$ are defined in the proof of Proposition \ref{prop:resolvent_KO}. 

By \eqref{eq:integral_resolvent}, we have
$$
P_{\lambda_i} g = \frac{-1}{2\pi {\rm i}} \int_{\gamma} R(\zeta, L_F) g {d}\zeta = \frac{1}{2\pi {\rm i}} \int_{\gamma} (\zeta I - L_F)^{-1} g {d}\zeta,
$$
where $\zeta$ is an complex integral curve enclosing only $\lambda_i$. 
Since $L_D:C^k (\mathbb{X}) \to (\mathbb{C}^n)^*$ is bounded, it follows 
$$
\begin{aligned}
L_D P_{\lambda_i} g &= \frac{1}{2\pi {\rm i}} L_D \int_{\gamma} (\zeta I - L_F)^{-1} g {d}\zeta \\
&= \frac{1}{2\pi {\rm i}} \int_{\gamma} L_D (\zeta I - L_F)^{-1} g {d}\zeta \\
&= \frac{1}{2\pi {\rm i}} \int_{\gamma} \Bigl[ L_D p_\zeta + L_D V_T (L_F - \zeta I)^{-1}|_B r_\zeta|_B \\ &\qquad\qquad~~ -L_D {\textstyle \int_0^T} {\rm e}^{-\zeta t} U_t r_\zeta dt \Bigr] {d}\zeta. 
\end{aligned}
$$
Now, for any $r\in \tilde{F}$, it can be shown that $(L_F - \zeta I)^{-1}|_B r|_B \in \tilde{F}|_B$, $V_T r \in \tilde{F}$, ${\textstyle \int_0^T} {\rm e}^{-\zeta t} U_t r \in \tilde{F}$, and that $L_D r = 0$. 
Therefore, we have 
\begin{equation}
\label{eq:lemma2}
L_D P_{\lambda_i} g = \frac{1}{2\pi {\rm i}} \int_{\gamma} L_D p_\zeta {d}\zeta.
\end{equation}
Next, applying $L_D$ to the both sides of \eqref{eq:lemma}, we have
$$
L_D g = (\zeta I - {\rm D}_x F (0)^* ) L_D p_\zeta 
$$
Here, we used 
$$\begin{aligned}
L_D L_F p_\zeta &= {\rm D}_x \bigl[ {\rm D}_x p_\zeta \cdot F \bigr]|_{x=0} = {\rm D}_x p_\zeta (0) {\rm D}_x F(0)\\ &= {\rm D}_x F(0)^* L_D p_\zeta     
\end{aligned}
$$
together with $L_D r_\zeta = 0$. 
Substituting this into \eqref{eq:lemma2}, we obtain
$$
\begin{aligned}
\frac{1}{2\pi {\rm i}} \int_{\gamma} L_D p_\zeta {d}\zeta &= \frac{-1}{2\pi {\rm i}} \int_{\gamma} ({\rm D}_x F (0)^* - \zeta I)^{-1} L_D g {d}\zeta  \\
&= \frac{-1}{2\pi {\rm i}} \int_{\gamma} R(\zeta,{\rm D}_x F (0)^*) L_D g {d}\zeta  \\ 
&= Q_i^*  L_D g = Q_i^* {\rm D}_xg(0)
\end{aligned}
$$
Combining this with \eqref{eq:lemma2} completes the proof. 
\end{proof}
}

We now proceed the proof of Theorem \ref{thm:linearization}. 

\begin{proof}
Here we denote $\sum_\ell P_{\lambda_\ell^u}^u := (P_{\lambda_1^u}+ \cdots + P_{\lambda_n^u})$. 
Define $\mathbb{E}^u = {\rm Range}\,\sum_\ell P_{\lambda_\ell^u}^u$ and $\mathbb{E} := \bigsqcup_{u \in \mathbb{U}} \{u \} \times \mathbb{E}^u$.
Furthermore, define $\Phi^u : \mathbb{E}^u \to (\mathbb{R}^n)^*$ by $\Phi^u (f):= {\rm D}_x f(0)$ and $\Phi:\mathbb{E} \to \mathbb{U} \times (\mathbb{R}^n)^*$ by $\Phi(u,f) := (u,{\rm D}_x f(0))$. 
Then, $\Phi^{-1} : \mathbb{U} \times (\mathbb{R}^n)^*  \to \mathbb{E}$ exists and is given by $\Phi^{-1} (u, (a_1, \ldots, a_n)) = (u, \sum_\ell P_{\lambda_\ell^u}^u (a_1 f_1 + \cdots + a_n f_n))$, where $f_i(x) := x_i$ for each $i\in \{1,\ldots, n\}$. 
From this, we see that $\Phi^u$ is a vector space isomorphism. 
Furthermore, according to Theorem \ref{thm:convergence_KEP_nonresonant}, the components of $\Phi$ are continuous in terms of $u$ in the $C^k(\mathbb{X})$ norm. 
Therefore, $\Phi:\mathbb{E} \to \mathbb{U}\times (\mathbb{R}^n)^*$ is a global trivialization of $\mathbb{E}$, implying $\mathbb{E}$ is a trivial bundle\cite[Chap. 10]{Lee_smooth}. 
Especially, defining $\psi^u_i := \sum_\ell P_{\lambda_\ell^u}^u f_i$ for each $i\in \{1,\ldots,n\}$, $(\psi_1^u, \ldots, \psi_n^u)$ is a global frame of $\mathbb{E}$\cite[Example 10.17]{Lee_smooth}, which indicates that $\{\psi_1^u, \ldots, \psi_n^u\}$ are linearly independent for all $u\in \mathbb{U}$. 
This and the invariance of $\mathbb{E}^u$ under the action of $L_F^u$ yield the existence of $A^u = \{a_{ij}^u \} \in \mathbb{R}^{n\times n}$ such that 
$$
L_{F}^u \psi_i^u (x) = a_{i1}^u \psi_1^u (x) + \cdots + a_{in}^u \psi_n^u (x) ,\quad \forall x\in \mathbb{X}, ~\forall u \in \mathbb{U}, 
$$
and 
$$
\frac{d}{dt} \psi^u (x) = A^u \psi^u (x) \quad \forall x \in \mathbb{X}, ~\forall u \in \mathbb{U}
$$
hold, where $\psi^u (x):= (\psi_1^u (x),\ldots, \psi_n^u (x))^\top$. 
We also see from Lemma \ref{lem:projection_differential} that ${\rm D}_x \psi^u (0) = I$ and from Remark \ref{rem:real-valued} that $\psi^u$ is real-valued. 
Moreover, Theorem \ref{thm:convergence_KEP_nonresonant} ensures that the components of $\psi^u$ are continuous in terms of $u$ in the $C^k (\mathbb{X})$ norm. 
The statement that $\psi^u$ is diffeomorphism is from \cite[Proposition 2]{kvalheim2021existence}. 

Finally, we show $A^u = {\rm D}_x F^u(0)$. 
The above deduction gives 
$$
\begin{aligned}
L_F^u \psi_i^u &= L_{F^u}{\textstyle \sum_\ell} P_{\lambda_{\ell}^u}^u f_i = (\lambda_1^u P_{\lambda_1^u}^u + \cdots + \lambda_n^u P_{\lambda_n^u}^u) f_i\\
&= a_{i1}^u {\textstyle \sum_\ell} P_{\lambda_{\ell}^u}^u f_1 + \cdots + a_{in}^u {\textstyle \sum_\ell} P_{\lambda_{\ell}^u}^u f_n . 
\end{aligned}
$$
Acting $L_D$ of Lemma \ref{lem:projection_differential} to the above equation, we obtain 
$$
\begin{aligned}
&(\lambda_1^u (Q_1^u)^* + \cdots + \lambda_n^u (Q_n^u)^* ) {\rm D}_x f_i (0) \\ 
&=  a_{i1}^u {\textstyle \sum_\ell} (Q_\ell^u)^* {\rm D}_x f_1 (0) + \cdots + a_{in}^u {\textstyle \sum_\ell}(Q_\ell^u)^* {\rm D}_x f_n (0)  \\
&=  a_{i1}^u {\rm D}_x f_1 (0) + \cdots + a_{in}^u {\rm D}_x f_n(0) \\
&=  ({\rm D}_x f_1 (0),\ldots, {\rm D}_x f_n(0)) (a_{i1}^u, \ldots, a_{in}^u)^\top
\end{aligned}
$$
where $(Q_i^u)^*$ is the dual operator of the eigenprojection of ${\rm D}_x F^u (0)$ associated with $\lambda_i^u$. 
Since the eigendecomposition of $({\rm D}_x F^u(0))^*$ yields $({\rm D}_x F^u(0))^* = \lambda_1^u (Q_1^u)^* + \cdots + \lambda_n^u (Q_n^u)^*$ and since ${\rm D}_x f_1 (0),\ldots, {\rm D}_x f_n(0)$ are the standard basis of $(\mathbb{R}^n)^*$, we have 
%$$\big({\rm D}_x F(0) \big)^* ({\rm D}_x f_1 (0), \ldots , {\rm D}_x f_n (0)) = ({\rm D}_x f_1 (0), \ldots , {\rm D}_x f_n (0)) A^\top$$
%and 
$A^u={\rm D}_x F^u(0)$. 
\end{proof}

\subsection{Proof of Theorem \ref{thm:bilinearization}}

\begin{proof}
Denote the Koopman principal eigenvalues of $L_F$ by $\lambda_1,\ldots, \lambda_n$ and the associated principal eigenfunctions by $\phi_{\lambda_1},\ldots,\phi_{\lambda_n}$. 
We first show that for any $X\in \{F,G_1,\ldots, G_d \}_{\mathfrak{L}}$ and any $j\in \{1,\ldots,n\}$, $L_X \phi_{\lambda_j} \in {\rm span} \{\phi_{\lambda_1},\ldots,\phi_{\lambda_n}\}$. 
Define the adjoint operator in terms of $A$ for the Lie subalgebra $\{A,B_1,\ldots, B_d \}_{\mathfrak{L}}$ by ${\rm ad}_A M := [A, M]$ for any $M\in \{A,B_1,\ldots, B_d \}_{\mathfrak{L}}$, and that in terms of $F$ for the Lie subalgebra $\{F,G_1,\ldots, G_d \}_{\mathfrak{L}}$ by ${\rm ad}_F X := [F, X]$ for any $X\in \{F,G_1,\ldots, G_d \}_{\mathfrak{L}}$. 
According to \cite[Proposition 2.1]{elliott2009bilinear}, the linear operator ${\rm ad}_A : \{A,B_1,\ldots, B_d \}_{\mathfrak{L}} \to \{A,B_1,\ldots, B_d \}_{\mathfrak{L}}$ has a point spectrum, which must consist of $\{ \lambda_k - \lambda_\ell~|~k,\ell \in \{1,\ldots,n \} \}$ where $\lambda_1,\ldots, \lambda_n$ are the eigenvalues of $A$. 
We introduce the eigenvector\footnote{While we assumed that $\lambda_k - \lambda_\ell$ is simple for simplicity, the proof can be generalized for the case that $\lambda_k - \lambda_\ell$ has a multiplicity. } associated with $\lambda_k - \lambda_\ell$ by $M_{k\ell}$. 
Let $\pi: \{A,B_1,\ldots, B_d \}_{\mathfrak{L}} \to \{F,G_1,\ldots, G_d \}_{\mathfrak{L}}$ be the isomorphism such that $\pi(A) = F$. 
Then we have 
$$
\begin{aligned}
&{\rm ad}_F \pi (M_{k\ell}) = [F,\pi (M_{k\ell})] 
= [\pi(A), \pi (M_{k\ell})] \\
&\qquad= \pi([A,M_{k\ell}]) = \pi({\rm ad}_A (M_{k\ell})) = \pi((\lambda_k-\lambda_\ell)M_{k\ell}) \\
&\qquad = (\lambda_k-\lambda_\ell) \pi (M_{k\ell})
\end{aligned}
$$
(indicating that $\pi(M_{k\ell})$ is the eigenvector of ${\rm ad}_F$), or, 
$$
(L_F L_{\pi (M_{k\ell})} - L_{\pi (M_{k\ell})} L_F )f = (\lambda_k-\lambda_\ell) L_{\pi (M_{k\ell})} f
$$
for all $f\in C^k(\mathbb{X})$. 
Setting $f=\phi_{\lambda_j}$, we have
$$
\big(L_F - (\lambda_j + \lambda_k - \lambda_\ell)I\big) L_{\pi (M_{k\ell})} \phi_{\lambda_j} = 0. 
$$
Since the eigenvalues of $A$ satisfy the $k$-nonresonant condition with $k\geq 2$, it follows
$$
{\rm Ker} \big(L_F - (\lambda_j + \lambda_k - \lambda_\ell)I\big) = \left\{ \begin{alignedat}{2}
    &{\rm span} \{ \phi_{\lambda_k}\} &\quad {\rm if}~j=\ell \\
    & \{0\} & \quad {\rm otherwise}
\end{alignedat} \right.
$$
which results in
$$
L_{\pi (M_{k\ell})} \phi_{\lambda_j} \in \left\{ \begin{alignedat}{2}
    &{\rm span} \{ \phi_{\lambda_k}\} &\quad {\rm if}~j=\ell~ \\
    &\{ 0 \} & \quad {\rm otherwise}.
\end{alignedat} \right.
$$
Since the eigenvectors $M_{k\ell}$ span $\{A,B_1,\ldots, B_d \}_{\mathfrak{L}}$ and $\pi$ is an isomorphism, 
$\pi(M_{k\ell})$ span $\{F,G_1,\ldots, G_d \}_{\mathfrak{L}}$. This implies that for any $X \in \{F,G_1,\ldots, G_d \}_{\mathfrak{L}}$ and any $j\in \{1,\ldots,n\}$, $L_X \phi_{\lambda_j} \in {\rm span} \{\phi_{\lambda_1},\ldots,\phi_{\lambda_n}\}$. 

Here, let $P\in \mathbb{C}^{n\times n}$ be an invertible matrix such that $AP =P{\rm diag}\{\lambda_1,\ldots, \lambda_n\}$ and define $\psi = (\psi_1,\ldots,\psi_n)^\top = P(\phi_{\lambda_1},\ldots, \phi_{\lambda_n})^\top$.
Then, the previous consideration implies the existence of matrices $C^i = \{c^i_{jk}\} \in \mathbb{C}^{n\times n}$ such that 
$$ \begin{aligned}
L_{G_i} \psi_j (x) &=  {\rm D}_x \psi_j (x) \cdot G_i (x) \\
&=c^i_{j1} \psi_1 (x) + \cdots + c^i_{jn} \psi_n (x). 
\end{aligned}
$$
for all $i=1,\ldots,n$. 
Differentiating in $x$ and substituting $x=0$, we have
$$
({\rm D}_x G_i(0))^* {\rm D}_x \psi_j (0) = c^i_{j1} {\rm D}_x \psi_1 (0) + \cdots + c^i_{jn} {\rm D}_x \psi_n (0), 
$$
where ${\rm D}_x \psi_1 (0), \ldots , {\rm D}_x \psi_n (0)$ span $(\mathbb{R}^n)^*$. 
This implies $(C^i)^* = ({\rm D}_x G_i(0))^* = (B_i)^*$. 
Finally, acting $L_{F+\sum_{i=1}^d u_i G_i}$ to $\psi$, we obtain \eqref{eq:bilinearization}. 
\end{proof}

\bibliographystyle{IEEEtran}
\bibliography{Refs}

@article{eldering2018global,
  title={Global linearization and fiber bundle structure of invariant manifolds},
  author={Eldering, Jaap and Kvalheim, Matthew and Revzen, Shai},
  journal={Nonlinearity},
  volume={31},
  number={9},
  pages={4202},
  year={2018},
  publisher={IOP Publishing}
}

@book{mauroy2020koopman,
  title={{The Koopman Operator in Systems and Control: Concepts, Methodologies, and Applications}},
  editor={Mauroy, Alexandre and Mezi{\'c}, Igor and Susuki, Yoshihiko},
  year={2020},
  publisher={Springer}
}

@article{mezic2020spectrum,
  title={Spectrum of the {K}oopman operator, spectral expansions in functional spaces, and state-space geometry},
  author={Mezi{\'c}, I.},
  journal={Journal of Nonlinear Science},
  volume={30},
  number={5},
  pages={2091--2145},
  year={2020},
  publisher={Springer}
}

@article{mezic2013analysis,
  title={Analysis of fluid flows via spectral properties of the {K}oopman operator},
  author={Mezi{\'c}, I.},
  journal={Annual Review of Fluid Mechanics},
  volume={45},
  pages={357--378},
  year={2013},
  publisher={Annual Reviews}
}

@article{kvalheim2021existence,
  title={Existence and uniqueness of global {K}oopman eigenfunctions for stable fixed points and periodic orbits},
  author={Kvalheim, M. D. and Revzen, S.},
  journal={Physica D: Nonlinear Phenomena},
  volume={425},
  pages={132959},
  year={2021},
  publisher={Elsevier}
}

@article{mezic2005spectral,
  title={Spectral properties of dynamical systems, model reduction and decompositions},
  author={Mezi{\'c}, Igor},
  journal={Nonlinear Dynamics},
  volume={41},
  pages={309--325},
  year={2005},
  publisher={Springer}
}

@book{Lee_smooth,
  title={{Introduction to Smooth Manifolds}},
  author={Lee, John M},
  year={2012},
  publisher={Springer}
}

@article{brunton2016koopman,
  title={Koopman invariant subspaces and finite linear representations of nonlinear dynamical systems for control},
  author={Brunton, Steven L and Brunton, Bingni W and Proctor, Joshua L and Kutz, J Nathan},
  journal={PloS one},
  volume={11},
  number={2},
  pages={e0150171},
  year={2016},
  publisher={Public Library of Science San Francisco, CA USA}
}

@book{kato2013perturbation,
  title={{Perturbation Theory for Linear Operators}},
  author={Kato, Tosio},
  volume={132},
  year={2013},
  publisher={Springer Science \& Business Media}
}

@article{liu2023non,
  title={On the non-existence of immersions for systems with multiple omega-limit sets},
  author={Liu, Zexiang and Ozay, Necmiye and Sontag, Eduardo D},
  journal={IFAC-PapersOnLine},
  volume={56},
  number={2},
  pages={60--64},
  year={2023},
  publisher={Elsevier}
}

@article{kvalheim2023linearizability,
  title={Linearizability of flows by embeddings},
  author={Kvalheim, Matthew D and Arathoon, Philip},
  journal={arXiv preprint arXiv:2305.18288},
  year={2023}
}

@article{lusch2018deep,
  title={Deep learning for universal linear embeddings of nonlinear dynamics},
  author={Lusch, Bethany and Kutz, J Nathan and Brunton, Steven L},
  journal={Nature communications},
  volume={9},
  number={1},
  pages={4950},
  year={2018},
  publisher={Nature Publishing Group UK London}
}

@article{susuki2021koopman_resolvent,
  title={{Koopman resolvent: A Laplace-domain analysis of nonlinear autonomous dynamical systems}},
  author={Susuki, Yoshihiko and Mauroy, Alexandre and Mezic, Igor},
  journal={SIAM Journal on Applied Dynamical Systems},
  volume={20},
  number={4},
  pages={2013--2036},
  year={2021},
  publisher={SIAM}
}

@article{irwin1972smoothness,
  title={On the smoothness of the composition map},
  author={Irwin, MC},
  journal={The Quarterly Journal of Mathematics},
  volume={23},
  number={2},
  pages={113--133},
  year={1972},
  publisher={Oxford University Press}
}

@phdthesis{banks2023koopman,
  title={Koopman Representations in Control},
  author={Banks, Michael James},
  year={2023},
  school={University of California, Santa Barbara}
}

@article{haseli2023modeling,
  title={{Modeling nonlinear control systems via Koopman control family: Universal forms and subspace invariance proximity}},
  author={Haseli, Masih and Cort{\'e}s, Jorge},
  journal={arXiv preprint arXiv:2307.15368},
  year={2023}
}

@article{guo2025learning,
  title={{Learning parametric Koopman decompositions for prediction and control}},
  author={Guo, Yue and Korda, Milan and Kevrekidis, Ioannis G and Li, Qianxiao},
  journal={SIAM Journal on Applied Dynamical Systems},
  volume={24},
  number={1},
  pages={744--781},
  year={2025},
  publisher={SIAM}
}

@article{sternberg1957local,
  title={{Local contractions and a theorem of Poincar{\'e}}},
  author={Sternberg, Shlomo},
  journal={American Journal of Mathematics},
  pages={809--824},
  year={1957},
  publisher={JSTOR}
}

@book{khalil2002nonlinear,
  title={{Nonlinear Systems}},
  author={Khalil, Hassan K},
  volume={3},
  year={2002},
  publisher={Prentice-Hall, Upper Saddle River, NJ.}
}

@article{proctor2018generalizing,
  title={Generalizing {K}oopman theory to allow for inputs and control},
  author={Proctor, Joshua L and Brunton, Steven L and Kutz, J Nathan},
  journal={SIAM Journal on Applied Dynamical Systems},
  volume={17},
  number={1},
  pages={909--930},
  year={2018},
  publisher={SIAM}
}

@article{shi2022deep,
  title={{Deep Koopman operator with control for nonlinear systems}},
  author={Shi, Haojie and Meng, Max Q-H},
  journal={IEEE Robotics and Automation Letters},
  volume={7},
  number={3},
  pages={7700--7707},
  year={2022},
  publisher={IEEE}
}

@inproceedings{susuki2024koopman,
  title={Koopman Resolvents of Nonlinear Discrete-Time Systems: Formulation and Identification},
  author={Susuki, Yoshihiko and Mauroy, Alexandre and Drma{\v{c}}, Zlatko},
  booktitle={2024 European Control Conference (ECC)},
  pages={627--632},
  year={2024},
  organization={IEEE}
}

@book{elliott2009bilinear,
  title={{Bilinear Control Systems: Matrices in Action}},
  author={Elliott, David LeRoy},
  volume={169},
  year={2009},
  publisher={Springer}
}

@article{goswami2021bilinearization,
  title={{Bilinearization, reachability, and optimal control of control-affine nonlinear systems: A Koopman spectral approach}},
  author={Goswami, Debdipta and Paley, Derek A},
  journal={IEEE Transactions on Automatic Control},
  volume={67},
  number={6},
  pages={2715--2728},
  year={2021},
  publisher={IEEE}
}

@article{strasser2025overview,
  title={{An overview of Koopman-based control: From error bounds to closed-loop guarantees}},
  author={Str{\"a}sser, Robin and Worthmann, Karl and Mezi{\'c}, Igor and Berberich, Julian and Schaller, Manuel and Allg{\"o}wer, Frank},
  journal={arXiv preprint arXiv:2509.02839},
  year={2025}
}

@article{proctor2016dynamic,
  title={Dynamic mode decomposition with control},
  author={Proctor, Joshua L and Brunton, Steven L and Kutz, J Nathan},
  journal={SIAM Journal on Applied Dynamical Systems},
  volume={15},
  number={1},
  pages={142--161},
  year={2016},
  publisher={SIAM}
}

@article{peitz2020data,
  title={{Data-driven model predictive control using interpolated Koopman generators}},
  author={Peitz, Sebastian and Otto, Samuel E and Rowley, Clarence W},
  journal={SIAM Journal on Applied Dynamical Systems},
  volume={19},
  number={3},
  pages={2162--2193},
  year={2020},
  publisher={SIAM}
}

@article{williams2015data,
  title={A data--driven approximation of the koopman operator: {Extending dynamic mode decomposition}},
  author={Williams, Matthew O and Kevrekidis, Ioannis G and Rowley, Clarence W},
  journal={Journal of Nonlinear Science},
  volume={25},
  number={6},
  pages={1307--1346},
  year={2015},
  publisher={Springer}
}

@article{gadginmath2024data,
  title={{Data-driven feedback linearization using the Koopman generator}},
  author={Gadginmath, Darshan and Krishnan, Vishaal and Pasqualetti, Fabio},
  journal={IEEE Transactions on Automatic Control},
  volume={69},
  number={12},
  pages={8844--8851},
  year={2024},
  publisher={IEEE}
}

@article{liu2025properties,
  title={{Properties of immersions for systems with multiple limit sets with implications to learning Koopman embeddings}},
  author={Liu, Zexiang and Ozay, Necmiye and Sontag, Eduardo D},
  journal={Automatica},
  volume={176},
  pages={112226},
  year={2025},
  publisher={Elsevier}
}

@article{cybenko1989approximation,
  title={Approximation by superpositions of a sigmoidal function},
  author={Cybenko, George},
  journal={Mathematics of control, signals and systems},
  volume={2},
  number={4},
  pages={303--314},
  year={1989},
  publisher={Springer}
}

@article{bruder2021advantages,
  title={{Advantages of bilinear Koopman realizations for the modeling and control of systems with unknown dynamics}},
  author={Bruder, Daniel and Fu, Xun and Vasudevan, Ram},
  journal={IEEE Robotics and Automation Letters},
  volume={6},
  number={3},
  pages={4369--4376},
  year={2021},
  publisher={IEEE}
}

\end{document}